\documentclass[10pt,a4paper]{amsart}
\usepackage[final]{optional}
\usepackage[british,american]{babel}

\usepackage{amsfonts, amsmath, wasysym}
\usepackage{verbatim}
\usepackage{amssymb,amsthm,
paralist
}
\usepackage{
latexsym,
nicefrac,
}
\usepackage[usenames]{color}

\usepackage{url}

\definecolor{darkgreen}{rgb}{0,0.5,0}
\definecolor{darkred}{rgb}{0.7,0,0}
\usepackage[colorlinks, 
citecolor=darkgreen, linkcolor=darkred
]{hyperref}
\usepackage{esint}
\usepackage{bibgerm}
\usepackage[normalem]{ulem}

\theoremstyle{plain}
\newtheorem{lemma}{Lemma}[section]
\newtheorem{thm}[lemma]{Theorem}
\newtheorem{prop}[lemma]{Proposition}
\newtheorem{cor}[lemma]{Corollary}

\theoremstyle{definition}

\newtheorem{rmk}[lemma]{Remark}

\numberwithin{equation}{section}

\newcommand{\de}{\delta}
\newcommand{\om}{\omega}
\newcommand{\Om}{\Omega}

\newcommand{\si}{\sigma}

\newcommand{\Si}{\Sigma}
\renewcommand{\th}{\theta}
\newcommand{\Th}{\Theta}

\newcommand{\R}{\ensuremath{{\mathbb R}}}

\newcommand{\Z}{\ensuremath{{\mathbb Z}}}
\newcommand{\C}{\ensuremath{{\mathbb C}}}

\newcommand{\downto}{\downarrow}

\newcommand{\tensor}{\otimes}

\DeclareMathOperator{\inj}{inj}
\newcommand{\beq}{\begin{equation}}
\newcommand{\eeq}{\end{equation}}
\newcommand{\beqs}{\begin{equation*}}
\newcommand{\eeqs}{\end{equation*}}
\newcommand{\beqa}{\begin{equation}\begin{aligned}}
\newcommand{\eeqa}{\end{aligned}\end{equation}}
\newcommand{\beqas}{\begin{equation*}\begin{aligned}}
\newcommand{\eeqas}{\end{aligned}\end{equation*}}
\newcommand{\brmk}{\begin{rmk}}
\newcommand{\ermk}{\end{rmk}}
\newcommand{\partref}[1]{\hbox{(\csname @roman\endcsname{\ref{#1}})}}

\newcommand*\arsinh{\mathop{\mathrm{arsinh}}\nolimits}

\newcommand{\norm}[1]{\Vert#1\Vert} 
\newcommand{\abs}[1]{\left\vert#1\right\vert}

\newcommand{\na}{\nabla}
\newcommand{\Qu}{\mathcal{Q}}
\newcommand{\Hol}{\mathcal{H}}
\newcommand{\pbz}{\partial_{\bar z}}
\newcommand{\thin}{\text{-thin}}
\newcommand{\thick}{\text{-thick}}
\newcommand{\Col}{\mathcal{C}}

\begin{document}
\title[ Poincar\'e estimate for quadratic differentials ]{A uniform Poincar\'e estimate for\\ quadratic differentials on closed surfaces}
%\thanks{20 April 2009.}

%\thanks{1991 Mathematics subject classification: }
\author{Melanie Rupflin and Peter Topping}
%\date{\today~(@\the\time mpm)}
\date{\today}
%\date{20 April 2009}

\begin{abstract}
We prove a uniform estimate, valid for every closed Riemann surface of genus at least two,
that bounds the distance of any quadratic differential to the finite dimensional space of holomorphic quadratic differentials 
in terms of its antiholomorphic derivative.
\end{abstract}
\maketitle
\section{Introduction}

Given any %smooth 
closed %\gcmt{orientable should be clear if we consider Riemann surfaces as changes in variables are holomorphic}
Riemann surface $(M,c)$, $c$ a complex structure, we consider the complex vector space $\Qu(M,c)$ of smooth quadratic differentials on $(M,c)$, 
that is of complex %$(0,2)$ 
tensors that with respect to a local complex coordinate $z$ take the form
%are given in complex coordinate charts $z$ as 
$$\Psi=\psi dz^2, \quad \psi \text{ a smooth function}.$$
Of particular importance is the subspace $\Hol(M,c)$ of those quadratic differentials that are represented in each complex coordinate chart 
by a holomorphic function. This space of holomorphic quadratic differentials has finite (complex) dimension 
$\dim(\Hol(M,c))=0$ for surfaces of genus $\gamma=0$, $\dim(\Hol(M,c))=1$ if $\gamma=1$ and 
$$\dim(\Hol(M,c))=3(\gamma-1) \text{ for } \gamma\geq 2 $$
by the Riemann-Roch theorem. It canonically represents the tangent space to Teichm\"uller space $\tau(M)$ 
at the point $[(M,c)]$, with $\Hol(M,c)$ equipped with the $L^2$ inner product,  
$$\langle \phi dz^2,\psi dz^2\rangle_{L^2(M,g)}=\int_M\phi\cdot\bar\psi \abs{dz^2}^2 dv_g=4\int_M\phi\cdot\bar\psi \cdot \rho^{-2} 
\frac{i}{2}dz\wedge d\bar z,$$ 
isometric 
to $T_{[(M,c)]}\tau(M)$ equipped with the Weil-Petersson metric. Here and in the following $g$
stands for the unique (modulo M\"obius transformations in the genus zero case) 
complete metric compatible with $c$ that has constant Gauss curvature $1,0,-1$ for surfaces of genus $0,1$ respectively $\gamma\geq 2$ (with unit area in the case $\gamma=1$) and $\rho$ denotes the
conformal factor corresponding to the complex coordinate $z$, determined by $g=\rho^2 dzd\bar z$.

Given any closed Riemann surface of finite genus $\gamma$ equipped with this canonical choice of metric
we now define
$$P_g:\Qu(M,c)\to \Hol(M,c)$$
to be the $L^2(M,g)$-orthogonal projection onto $\Hol(M,c)$.

Furthermore we denote by $\bar\partial \Psi$ the antiholomorphic derivative of a quadratic differential $\Psi$,
that is the 
%$(0,3)$ 
tensor given in complex coordinates by 
$$\bar\partial\Psi=\partial_{\bar z}\psi d\bar z \tensor dz^2.$$

In this paper we prove an estimate for arbitrary quadratic differentials that is reminiscent of the standard Poincar\'e inequality for functions
\beq \label{est:standard} \norm{f-\bar f}_{L^1}\leq C \cdot \norm{\na f}_{L^1}\eeq
bounding the distance of an object from its projection onto a finite dimensional subspace, here the constant functions, in terms of a derivative. %, in the case of \eqref{est:standard} with the constant depending on the domain $\Om$.

However, in stark contrast to the standard Poincar\'e inequality for functions, our inequality for quadratic differentials is 
\textit{uniform}, i.e. independent of the 
geometry of the hyperbolic surface on which we work;
it is valid with a constant $C$ depending only on the topology of the surface (i.e. on the genus) and this feature is essential for applications to the Teichm\"uller harmonic map flow \cite{R-T} to which we allude briefly below.
One further distinction between the normal Poincar\'e estimate and the new estimate is that we use the $\bar \partial$ operator on the right-hand side rather than the full derivative $\nabla$, which makes our estimate also an elliptic estimate, and means that we should make estimates relative to its kernel (i.e. the holomorphic quadratic differentials) rather than the kernel of $\nabla$ (i.e. the constant functions).

%As the main result we prove
\begin{thm}\label{thm:Poincare} (Main theorem.)
Given an arbitrary closed Riemann surface 
$(M,c)$ of genus at least two, there exists a constant $C<\infty$ depending only on the genus of $M$ such that the following uniform Poincar\'e estimate holds true.
The distance of any quadratic differential $\Psi\in\Qu(M,c)$ from its holomorphic part is uniformly bounded in terms of its antiholomorphic derivative in the sense that
\beq\label{est:Poincare} \norm{\Psi-P_g(\Psi)}_{L^1(M,g)}\leq C\cdot \norm{\bar \partial\Psi}_{L^1(M,g)}.\eeq
Here and in the following all norms are computed with respect to the unique hyperbolic  metric $g$ compatible with $(M,c)$.
\end{thm}

\begin{rmk}
While the left-hand side of \eqref{est:Poincare} is invariant under a conformal change of the metric, the right hand side is not. 
It is important here to take the unique hyperbolic conformal metric.
\end{rmk} 

%\begin{rmk}
For hyperbolic surfaces contained in a compact region of moduli space, 
the estimate \eqref{est:Poincare} was shown in \cite{R-T}, Lemma 2.1.
%\end{rmk}
In that paper, the estimate was used to understand the asymptotics of a new flow which deforms a pair $(u,g)$, where $u$ is a map from a closed surface $M$ to an arbitrary compact Riemannian manifold and $g$ is a constant curvature $-1$ metric on $M$, under the gradient flow of the harmonic map energy. Where possible, the flow tries to converge to a branched minimal immersion, and a key to demonstrating this is to argue that the Hopf differential $\Phi$ has not only 
$\bar\partial \Phi$ and $P_g(\Phi)$ converging to zero in $L^1$, but also that $\Phi$ itself converges to zero in $L^1$. The estimate of
Theorem \ref{thm:Poincare} immediately implies this key fact.

The space of holomorphic quadratic differentials on a surface of genus $0$, i.e. on a sphere, is trivial so that in this case the Poincar\'e estimate 
takes the following form, and can be proved with standard techniques
(cf. \cite[Lemma 2.5]{annals}).

\begin{prop}\label{prop:sphere}
There exists a constant $C<\infty$ %(may for example be chosen as $C=4$) 
such that all quadratic differentials $\Psi\in\Qu(S^2,g)$ on the sphere satisfy
\beq\label{est:sphere}
\norm{\Psi}_{L^1(S^2,g)}\leq C\norm{\bar \partial\Psi}_{L^1(S^2,g)},\eeq
where $g$ is the metric of constant curvature $1$.% on the sphere.
\end{prop}

\begin{rmk}
The moduli space of the sphere, that is the set of equivalence classes of complex structures that agree up to pull-back by an orientation-preserving diffeomorphism, consists of only one point. 
Since the estimate \eqref{est:sphere} (and also \eqref{est:Poincare}) is invariant under the pull-back by diffeomorphisms, Proposition \ref{prop:sphere} essentially just says that a Poincar\'e estimate is valid for the 
(unique) complex structure on the sphere.
\end{rmk}

Similarly, such an estimate is valid for every fixed complex structure of a torus. Contrary to surfaces of higher genus however, this estimate is \textit{not} uniform.
\begin{prop}
For any flat unit area torus $(T^2,g)$ there exists $C<\infty$ such that 
\beq\label{est:torus}\norm{\Psi-P_g(\Psi)}_{L^1(T^2,g)}\leq C \norm{\bar\partial\Psi}_{L^1(T^2,g)} \text{ for every } \Psi\in \Qu(T^2,g).\eeq
This estimate is \textit{not} uniform; given any number $C<\infty$ there exists a torus $(T^2,g)$ that is flat and has unit area but for which \eqref{est:torus} is violated.
\end{prop}

Indeed note that each such torus is isometric to $(\C/\Gamma_{a,b}, g_{eucl})$ for some lattice group $\Gamma_{a,b}=\{n\cdot b+m\cdot (a+\frac{i}{b}),\, n,m\in\Z\}$ with $a\in\R$ and $b>0$, 
and that $\Phi=\phi dz^2$ is a holomorphic
quadratic differential if and only if $\phi$ is a $\Gamma$-periodic holomorphic function, i.e.~a constant. Thus in this special case the Poincar\'e estimate for quadratic differentials
is equivalent to a refined Poincar\'e estimate for $\Gamma$-periodic functions of
$$\norm{\phi-\bar\phi}_{L^1(\C/\Gamma)}\leq C\cdot \norm{\pbz \phi}_{L^1(\C/\Gamma)}.$$
A simple example, say $\phi(x+iy)=\sin(\frac{2\pi}{b}x)$ on $\C/\Gamma_{a=0,b}$ with $b\to \infty$, shows that such an estimate is not uniform.

Theorem \ref{thm:Poincare} will be proved by contradiction. If the result were not true, then we would find a sequence of surfaces and quadratic differentials (which without loss of generality would have no holomorphic part at all)
% (i.e. their projection under $P$ would be null)
which violate \eqref{est:Poincare} for larger and larger values of $C$. The surfaces would have to degenerate by pinching certain necks, because otherwise the result is known from \cite{R-T}. After normalising so that the $L^1$ norm is always $1$, we then pass to a subsequence to get a noncompact limit surface together with a limit quadratic differential which will be holomorphic (see Lemma \ref{lemma:convergence}). But a result in \cite{R-T-Z} (see Lemma \ref{lemma:projection}) tells us that the limit quadratic differential inherits the property of having no holomorphic part at all, and thus must be identically zero.
The key part of this paper is then to show that the limit inherits the property of having $L^1$ norm equal to $1$, giving a contradiction. The essential point is that we must prove that in this limit, $L^1$ norm cannot concentrate on degenerating collars and be lost in the limit, and this is articulated by our key Lemma \ref{lemma:L1}. Essentially, the only way that $L^1$ norm of an almost-holomorphic quadratic differential can concentrate on a long collar is if the quadratic differential has a nonvanishing `principal component' -- i.e. its lowest Fourier mode on the collar is not disappearing, and this component looks like a parallel quadratic differential on the collar. However, by assumption our quadratic differentials are orthogonal to all holomorphic quadratic differentials, and in Lemma \ref{lemma:basis} we construct a sequence of holomorphic quadratic differentials which is purely concentrating on the collar, and is becoming parallel. Thus our original sequence cannot concentrate $L^1$ norm on the collar as desired.

{\em Acknowledgements:} Partially supported by The Leverhulme Trust.

\section{Proof of the main result}
The basic strategy of the proof of Theorem \ref{thm:Poincare} is similar to the one in \cite{R-T} in that we argue by contradiction and use compactness results in order to pass in the limit to a holomorphic quadratic differential 
on some limit surface. A key difference is however that in order to obtain the \textit{uniform} version of the Poincar\'e estimate claimed in Theorem \ref{thm:Poincare} 
we need to be able to deal with degenerating sequences of surfaces, with the local arguments of \cite{R-T} only applicable for considerations away from the degenerating parts of these surfaces.
A crucial part of the proof is thus a discussion, from the point of view of geometric analysis, first of holomorphic quadratic differentials 
on a sequence of degenerating hyperbolic surfaces, and then, more generally, of
non-holomorphic quadratic differentials with controlled antiholomorphic derivatives.

Contrary to the assertion of the theorem, 
let us suppose that there exist a sequence of closed hyperbolic surfaces $(M_i,c_i,g_i)$ of fixed genus, and a sequence of (nonholomorphic) quadratic differentials 
$\Phi_i$ on $(M_i,c_i)$ such that 
$$\frac{\norm{P_{g_i}(\Phi_i)-\Phi_i}_{L^1(M_i,g_i)}}{\norm{\bar\partial\Phi_i}_{L^1(M_i,g_i)}}\to\infty.$$
Replacing $\Phi_i$ by a (multiple of) $P_{g_i}(\Phi_i)-\Phi_i$, 
using the uniformisation theorem and 
pulling back by an appropriate family of diffeomorphisms from $M$ to $M_i$ we obtain the following setting: 

\textit{Assumptions:}
We assume that there exists a closed surface $M$ of genus $\gamma\geq 2$ such that there is a sequence of complex structures $c_i$ on $M$ and a sequence of quadratic differentials 
$\Phi_i\in\Qu(M,c_i)$ for which the following three assumptions are true:
\beq \label{ass}
  %\begin{align}
       P_{g_i}(\Phi_i)=0 \quad\text{ and }\quad%\label{ass:Pg}\\ %\label{ass:L1} 
\norm{\Phi_i}_{L^1(M,g_i)}=1 \quad\text{ and }\quad%\label{ass:dzbar}
\norm{\bar\partial\Phi_i}_{L^1(M,g_i)}\to 0 \text{ as } i\to\infty.
  %\end{align}
%\end{subequations}
\eeq
Here and in the following $g_i$ stands for the unique complete hyperbolic metric compatible with the complex structure $c_i$.

Since we know \cite{R-T} that the Poincar\'e estimate \eqref{est:Poincare} is valid on every compact subset $K$ of moduli space 
(with a constant $C$ depending a priori on $K$) the surfaces $(M,g_i)$ must degenerate in moduli space, i.e. the length of the shortest closed geodesic of $(M,g_i)$ must converge to 
zero as $i\to\infty$.

According to the Deligne-Mumford compactness theorem \cite{Del-Mum}, after passing to a subsequence we may assume that 
$(M,g_i)$ degenerates to a hyperbolic punctured surface $(\Si,h)$ 
(i.e. a surface obtained from finitely many closed Riemann surfaces by removing finitely many points, which is equipped with the complete hyperbolic metric that is compatible with the induced complex structure)
by collapsing $1\leq k\leq 3(\gamma-1)$ 
geodesics. In practice this means that there exist simple closed geodesics $\{\si_i^j\}_{j=1}^k$ on $(M,g_i)$ of length $\ell(\si_i^j)\to 0$ as $i\to \infty$ 
and diffeomorphisms $f_i:\Si\to M\setminus \cup_{j=1}^k \si_i^j$ 
such that the metrics and the corresponding complex structures converge
$$f_i^*g_i\to h,\quad f_i^*c_i\to c_\infty\,\text{ smoothly locally on } \Si.$$
Here the limiting surface $(\Si,c_\infty,h)$ is a non-compact, possibly disconnected, complete hyperbolic surface with $2k$ punctures corresponding to the collapsing geodesics in the sense that $f_i^{-1}$ extends
to a continuous map from $M$ to the compactification of $(\Si,h)$ obtained by filling in $k$ appropriate pairs of punctures with $k$ new points; each geodesic $\si_i^j$ is then mapped by $f_i^{-1}$ to a different one of these (paired) points.

In this situation we then derive a contradiction from the assumptions in \eqref{ass} in three steps; first, and using only local arguments similar to the ones of \cite{R-T},
we obtain that a subsequence of $f_i^*\Phi_i$ converges \textit{locally} to a holomorphic limit $\Phi_\infty$;
second, we find that $\Phi_\infty$ stands orthogonal to the space of integrable holomorphic quadratic differentials on the limit surface, so that the holomorphic quadratic differential 
$\Phi_\infty$ obtained in the first step must be identically zero. 
Finally, we will show that despite the convergence of the $f_i^*\Phi_i$ being only local, the $L^1$ norm is preserved globally in the limit $i\to \infty$ 
and thus that $\norm{\Phi_\infty}_{L^1(\Si,h)}=1$ in contradiction to $\Phi_\infty\equiv 0$.

\begin{lemma}\label{lemma:convergence}
 Let $(M,g_i)$ be a sequence of closed hyperbolic surfaces that degenerates to a hyperbolic punctured surface $(\Si,h)$ as described above. Then for any sequence of 
quadratic differentials $\Psi_i\in\Qu(M,g_i)$ with
$$\norm{\Psi_i}_{L^1(M,g_i)}+\norm{\bar \partial\Psi_i}_{L^1(M,g_i)}\leq C<\infty$$
there exists a subsequence converging 
$$f_i^*\Psi_i\to \Psi_\infty \text{ in } L_{loc}^1(\Si,h)$$
to a quadratic differential $\Psi_\infty\in L^1(\Sigma,h)$.
Additionally if $\norm{\bar \partial \Psi_i}_{L^1(\Si,h)}\to 0$ then $\Psi_\infty$ is holomorphic.
\end{lemma}

Given a sequence $\Phi_i$ as in \eqref{ass} we thus find that after passing to a subsequence and pulling-back by diffeomorphisms it converges to a limit $\Phi_\infty$ which is 
an element of the space 
$$\Hol(\Si,h):=\{ \Psi \text{ a holomorphic quadratic differential on }(\Si,h) \text{ with } \norm{\Psi}_{L^1(\Si,h)}<\infty\}$$
which can be equivalently characterised as the space of holomorphic quadratic differentials with at most a simple pole at each puncture. 
In the limit $i\to\infty$ the dimension of $\Hol$ reduces by the number $k$ of collapsing geodesics, i.e. $\dim_\C(\Hol(\Si,h))=3(\gamma-1)-k=\dim_\C(\Hol(M,g_i))-k$, by Riemann-Roch.

As is discussed in \cite{R-T-Z}, Lemma A.11, the definition of $\Hol(\Si,h)$ would be unchanged if we required $\norm{\Psi}_{L^\infty(\Si,h)}<\infty$
instead of $\norm{\Psi}_{L^1(\Si,h)}<\infty$. (Note that the volume of $(\Si,h)$ is finite, so the $L^\infty$ norm controls the $L^1$ norm. On the other hand, controlling the $L^1$ norm gives sufficient control on the order of any poles for the $L^\infty$ norm also to be controlled.)
In particular, all elements in $\Hol(\Si,h)$ lie in $L^2$, and 
the space of $L^2$ quadratic differentials then has a natural projection onto $\Hol(\Si,h)$. Moreover, because every element of $\Hol(\Si,h)$ lies in $L^\infty$, this projection extends naturally to the space of $L^1$ quadratic differentials.

It thus makes sense to analyse the projection onto $\Hol(\Si,h)$ of any quadratic differential $\Psi_\infty$ obtained as a local $L^1(\Si,h)$ limit of a sequence of quadratic differentials with uniformly bounded
$L^1$ norm, and in particular to assert that the limit 
$\Phi_\infty$ of the sequence $\Phi_i$ satisfying \eqref{ass} is 
orthogonal to $\Hol(\Si,h)$. 
(Together with our knowledge that $\Phi_\infty$ is holomorphic, this will imply that $\Phi_\infty\equiv 0$.)
In order to prove this, we make use of the following continuity result for the projections onto the spaces $\Hol(\cdot)$,
derived in \cite{R-T-Z}. A precise definition of the spaces $W_i$ will be given later once we have some more notation.

\begin{lemma}\label{lemma:projection}
Let $(M,g_i)$ be any sequence of closed hyperbolic surfaces that degenerates to a hyperbolic punctured surface $(\Si,h)$ by collapsing $k$ collars. Then there exist subspaces $W_i\subset \Hol(M,g_i)$
of dimension $3(\gamma-1)-k$ such that the $L^2(M,g_i)$-orthogonal projections $P_{g_i}^{W_i}$ onto $W_i$ converge to the $L^2(\Si,h)$-orthogonal projection $P_h^{\Hol(\Si,h)}$ onto the space of integrable holomorphic 
quadratic differentials $\Hol(\Si,h)$ in the following sense:

For any sequence $\Psi_i\in\Qu(M,g_i)$ of quadratic differentials on $(M,g_i)$ with $\norm{\Psi_i}_{L^1(M,g_i)}$ bounded and with $f_i^*\Psi_i\to \Psi_\infty$
locally in $L^1(\Si,h)$, we have
$$f_i^*\big(P_{g_i}^{W_i}(\Psi_i)\big)\to P_h^{\Hol(\Si,h)}(\Psi_{\infty}) \text{ smoothly locally on }\Si.$$ 
\end{lemma}
We remark that a stronger statement holds true for the projections $P_{g_i}^{W_i}$, see Theorem 2.6 of \cite{R-T-Z},
asserting not only local convergence 
but also convergence of the \textit{global} $L^p$ norms, $1\leq p\leq \infty$, to the corresponding norm of the limit  
and thus excluding any concentration of elements of $W_i$ on the degenerating parts of the surface. 
Conversely, we will later see that elements of $W_i^\perp$ concentrate solely on \textit{degenerating collar regions}.

Returning to our sequence of quadratic differentials $\Phi_i$ note that \eqref{ass} implies in particular that $P_{g_i}^{W_i}(\Phi_i)=0$, so by Lemma \ref{lemma:projection} the limit
$\Phi_\infty\in\Hol(\Si,h)$ obtained above must satisfy 
$P_h^{\Hol(\Si,h)}(\Phi_\infty)=0$ and thus vanish identically $\Phi_\infty\equiv 0$.

Conversely, we will prove that despite the convergence of $f_i^*\Phi_i$ on $\Si$ being only local, 
the $L^1$ norms are preserved globally and thus $\norm{\Phi_\infty}_{L^1(\Si,h)}=1$. In this argument, the key point to be proven is that there can be no concentration of $L^1$ norm on 
the degenerating parts of the surface. 
This follows from the following more general result, controlling almost-holomorphic quadratic differentials in the $\de$-thin part of the surface (i.e. where the injectivity radius is less than $\de$) which is central to this paper.

\begin{lemma}\label{lemma:L1}
Let $(M,g)$ be any closed hyperbolic surface. Then there exists a constant $C<\infty$  depending only on the genus of the surface $M$ such that for every quadratic differential $\Psi\in\Qu(M,g)$ with 
$$P_g\Psi=0,$$
and every $\de>0$, we have the estimate 
$$\norm{\Psi}_{L^1(\de\text{-thin}(M,g))}\leq C\cdot \big(\norm{\bar \partial \Psi}_{L^1(M,g)}+\de^{1/2}\norm{\Psi}_{L^1(M,g)}\big).$$
\end{lemma}

Returning to our sequence of quadratic differentials $\Phi_i$ satisfying the assumptions in \eqref{ass} we thus find that the $L^1$ norms on the $\de$-thick parts 
$\Si_i^\de:=\{p\in\Si: \inj_{f_i^*g_i}(p)\geq \de\}$ of the degenerating surfaces $(\Si,f_i^*g_i)$ satisfy
\beq\label{est:sup-L1}
\sup_{\de>0}\lim_{i\to \infty}\norm{f_i^*\Phi_i}_{L^1(\Si_i^\de,f_i^*g_i)}=1-\inf_{\de>0}\lim_{i\to\infty}\norm{\Phi_i}_{L^1(\de\text{-thin}(M,g_i))}=1.\eeq

We remark that the special structure of hyperbolic surfaces, in particular the collar lemma of Keen-Randol \cite{randol}, %and the characterisation of \cite{Hu} of the $\arsinh(1)$-thin set as a union such collar around closed geodesics,
leads to the observation that for any $0<\de<\arsinh(1)$ the $\de$-thick part 
$\Si_i^\de$ of $(\Si,f_i^*g_i)$ converges to the 
$\de$-thick part $\Si^\de$ of the limiting surface $(\Si,h)$, which is a compact subset of $(\Si,h)$, 
in the sense that both $\Si_i^\de$ and $\Si^\de$ lie in a fixed ($i$-independent) compact set and the measure of their symmetric difference converges to zero (see Lemma A.7 
in \cite{R-T-Z}). 
Combined with the locally uniform convergence of the metrics $f_i^*g_i\to h$ we thus conclude that for 
any $\de>0$ 
\beq \label{eq:L1-thick}
\norm{\Phi_\infty}_{L^1(\Si^\de,h)}=\lim_{i\to\infty}\norm{f_i^*\Phi_i}_{L^1(\Si_i^\de,f_i^*g_i)}
\eeq
so taking the supremum over $\de>0$ and using \eqref{est:sup-L1} we must have $\norm{\Phi_\infty}_{L^1(\Si,h)}=1$ 
in contradiction to the fact that $\Phi_\infty\equiv 0$ which was a consequence of Lemmas \ref{lemma:convergence} and \ref{lemma:projection}.

\begin{proof}[Proof of Lemma \ref{lemma:convergence}]
The main tool for the proof of this lemma is the compactness lemma 2.3 of \cite{R-T} for functions on the euclidean disc $D_1$ for which the $L^1$ norm of both the function and its antiholomorphic derivative is bounded.

Let $\Psi_i$ and $(M,g_i)$ be as in Lemma \ref{lemma:convergence} and recall that the convergence of the metrics $f_i^*g_i$ also implies convergence of the associated complex 
structures $f_i^*c_i$.
In practice this means that given any compact subset $K\subset \Si$ there exists a number $\de>0$ and 
a sequence of atlases covering $K$ which consist of coordinate charts that can be viewed as isometries
%a converging sequence of what we call \emph{isometric coordinate charts}
$$\phi_i^j:B_{f_i^*g_i}(p^j,\de)\to(B_{g_H}(0,\de),g_H),$$
from the balls $B_{f_i^*g_i}(p^j,\de)$ of radius $\de$ in $(\Si,f_i^*g_i)$ to the fixed ball $B_{g_H}(0,\de)$ of radius $\de$ in the Poincar\'e hyperbolic disc,
and the maps $\phi_i^j$ converge smoothly to an isometry $\phi_\infty^j$
from $B_{h}(p^j,\de)\subset (\Si,h)$ to 
$(B_{g_H}(0,\de),g_H)$. Here we can assume that for each $i$, the set $K$ is covered not only by $B_{f_i^*g_i}(p^j,\de)$ but also by the balls $B_{f_i^*g_i}(p^j,\de/2)$ with half the radius. 

The assumptions of Lemma \ref{lemma:convergence} then imply uniform $L^1(B_{g_H}(0,\de))$ bounds on both the functions $\psi_i^j$ representing $f_i^*\Psi_i$ in these coordinate charts, 
and their antiholomorphic derivatives. Thus applying Lemma 2.3 of \cite{R-T} and passing to a subsequence we find that the functions $\psi_i^j$ converge in 
$L^1$ on a slightly smaller disc, say on $B_{g_H}(0,\de/2)$, to a limit $\psi_\infty^j$ and that this limit is holomorphic if 
$\norm{\pbz\psi_i^j}_{L^1(B_{g_H}(0,\de))}\to 0$ and thus in particular if
$\norm{\bar \partial \Psi_i}_{L^1(M,g_i)}\to 0$.

Pulling back by the charts $\phi_i^j$ as well as making use of the convergence of the metrics we then obtain that the quadratic differentials $f_i^*\Psi_i$ converge 
to a limiting quadratic differential $\Psi_\infty$ in the sense of $L^1(K,h)$ convergence of tensors and that the limit is holomorphic provided the antiholomorphic derivatives 
converge to zero as described in the lemma. 
\end{proof}

The main step of the proof of Theorem \ref{thm:Poincare} thus consists in proving Lemma \ref{lemma:L1}.

\textit{Proof of Lemma \ref{lemma:L1}.}
Let $(M,g)$ be a closed hyperbolic surface of genus $\gamma$. We recall that the Keen-Randol collar lemma \cite{randol} gives the 
following explicit description of $(M,g)$ near each simple closed geodesic $\si$ of length 
$\ell>0$: 
there is a neighbourhood $\Col$ of $\si$ in $(M,g)$ which is isometric to $\Col(\ell)$, where
$\Col(\ell)$ is the cylinder $(-X(\ell), X(\ell)) \times S^1$ equipped with the metric $\rho^2(s)(ds^2+d\th^2)$, with
$$X(\ell)= \frac{2\pi}{\ell}\left(\frac{\pi}{2}-\arctan\left(\sinh\left(\frac{\ell}{2}\right)\right) \right),  \quad\text{ and }\quad \rho(s)=\frac{\ell}{2\pi \cos(\frac{\ell s}{2\pi})}.$$
We will also use on several occasions that for $z=s+i\th$, 
\begin{equation}
\label{sizes_on_collars}
|dz^2|=2\rho^{-2}, \qquad\text{ and }\qquad \norm{dz^2}_{L^2(\Col(\ell))}^2 = 8\pi\int_{-X}^X\rho^{-2}(s)ds\sim \ell^{-3},
\end{equation}
as $\ell\downto 0$.

Further important results in the theory of hyperbolic surfaces, cf. \cite{Bu}, Theorems 4.1.1 and 4.1.6, 
tell us that the collar regions around geodesics of length $0<\ell<2\arsinh(1)$ are disjoint,
the number of closed geodesics $\si^j$ of length less than $2\arsinh(1)$ %\gcmt{factor 2?} 
is no more than $3(\gamma-1)$, where $\gamma$ is the genus of $M$, and that for any 
$0<\de<\arsinh(1)$ the $\de$-thin part of any 
hyperbolic surface consists solely of (subcylinders of) such collars $\Col^j$. Furthermore 
the $\de$-thin part of each such a collar $\Col(\ell)$, $0<\ell<2\arsinh(1)$ is given by the subcylinder 
%$\Col^j(\de)=
$(-X_\de(\ell),X_\de(\ell))\times S^1$ for 
\beq \label{eq:Xdelta}
X_\de(\ell)=\frac{\pi^2}{\ell}-\frac{2\pi}{\ell}\arcsin\left(\frac{\sinh(\ell/2)}{\sinh(\de)}\right)
\eeq if $\de\in(\ell/2,\arsinh(1))$ respectively by the empty set if $\de\leq \ell/2$. We also remark that using this formula 
one can easily check the intuitively clear fact that $\rho(X_\de)$ is of order $\de$, we write for short $\rho(X_\de)\sim \de$, in the sense that 
there is a constant $C<\infty$ such that for every collar $\Col(\ell)$, $0<\ell<2\arsinh(1)$ and every $\de\in(\ell/2, \arsinh(1))$
%the estimate 
we have
\beq 
\label{est:rho-Xde}
C^{-1}\de\leq \rho(X_\de(\ell))\leq C\cdot \de.
\eeq
%holds true. 

%
In order to prove Lemma \ref{lemma:L1} we need to show that an estimate of the form 
\beqa\label{est:L^1collar} 
\norm{\Psi}_{L^1(\de\text{-thin}(\Col))}%=\int_{-X_\de(\ell)}^{X_\de(\ell)}\int_{S^1} \abs{\psi}\cdot\abs{dz^2} \rho^2 dsd\th
=2\int_{-X_\de(\ell)}^{X_\de(\ell)}\int_{S^1} \abs{\psi} d\th ds
\leq C(\de^{1/2}\norm{\Psi}_{L^1(M,g)}+\norm{\bar \partial \Psi}_{L^1(M,g)})\eeqa
is valid for each such collar $\Col\cong\Col(\ell)$ and each quadratic diffential $\Psi\in\Qu(M,g)$ satisfying $P_g\Psi=0$. Here and in the following $C$ denotes a constant depending only 
on the genus of the 
surface $M$, which we assume to be fixed.

We prove this claim in two steps. First we show that the assumed orthogonality of $\Psi$ to $\Hol(M,g)$ implies estimates on mean values on circles of 
the function $\psi=\psi(s,\th)$ representing $\Psi$ in the collar coordinates $(s,\th)$.
\begin{lemma}\label{lemma:alpha}
Let $(M,g)$ be any closed hyperbolic surface. Then there exists a constant $C<\infty$  depending only on the genus of $M$ such that for
any collar region $\Col=\Col(\ell)$, $\ell>0$, of $(M,g)$ as described above and 
any quadratic differential $\Psi=\psi dz^2\in\Qu(M,g)$ satisfying 
$$P_g\Psi=0,$$ 
the mean values 
$$\alpha(s):=\frac1{2\pi}\int_{S^1 }\psi(s,\th) d\th,\quad s\in (-X(\ell),X(\ell))$$
(where we work with respect to the local complex coordinate $z=s+i\th$ on the collar)
satisfy the estimate
\beq
\int_{-X(\ell)}^{X(\ell)}\abs{\alpha(s)} ds\leq C\cdot \big(\norm{\bar \partial \Psi}_{L^1(M,g)}+ \ell\cdot\norm{\Psi}_{L^1(M,g)}\big).\label{est:alpha-lemma}\eeq
\end{lemma}

In a second step, which will complete the proof of Lemma \ref{lemma:L1}, we then estimate the $L^1$ norm of general quadratic differentials on the thin part of a collar in terms of $\alpha(\cdot)$ and the antiholomorphic derivative.
\begin{lemma}\label{lemma:L1-est-using-alpha}
Let $(M,g)$ be any closed hyperbolic surface. Then there exists a constant $C<\infty$  depending only on the genus of $M$ such that for
any collar region $\Col$, 
any quadratic differential $\Psi\in\Qu(M,g)$ and any $\de>0$, we have
$$\norm{\Psi}_{L^1(\de\text{-thin}(\Col))}\leq C\cdot\big(\int_{-X}^{X}\abs{\alpha(s)} ds+ \norm{\bar\partial \Psi}_{L^1(M,g)}+\delta^{1/2}\cdot\norm{\Psi}_{L^1(M,g)}\big).$$
\end{lemma}
Since the $\de$-thin part of any collar $\Col(\ell)$ with $\ell>2\de$ is the empty set, and since the claim of Lemma \ref{lemma:L1} is trivially true for 
large values of $\de$, combining Lemmas \ref{lemma:alpha} and \ref{lemma:L1-est-using-alpha} immediately gives the claim of the key Lemma \ref{lemma:L1}.

Before we prove Lemma \ref{lemma:alpha}, we remark that the estimate claimed in that lemma remains valid if we 
weaken the assumption of $\Psi$ being orthogonal to the whole space $\Hol(M,g)$, and only demand $\Psi$ to be orthogonal 
to one specific holomorphic quadratic differential (per collar), which is described as follows:
\begin{lemma}\label{lemma:basis}
Let $(M,g)$ be any closed hyperbolic surface. Then there exists a constant $C<\infty$  depending only on the genus of $M$ such that for
$\si$ any closed
geodesic of length $0<\ell\leq 2\arsinh(1)$ and $\Col$ its collar neighbourhood described above, there exists a holomorphic
quadratic differential $\Om$ with $\norm{\Om}_{L^2(M,g)}=1$, concentrated on this one collar in the sense that 
\beq\label{est:Om-complement}
\norm{\Om}_{L^\infty(M\setminus \Col,g)}\leq C\ell^{1/2},
\eeq
and on this collar essentially given as a constant multiple of $dz^2$ in the sense that
\beq\label{est:Om-collar} 
\norm{\Om-b_0dz^2}_{L^\infty(\Col,g)}\leq C \ell^{1/2},
\eeq
for a number $b_0\in\C$ satisfying $\abs{1-\abs{b_0}\cdot \norm{dz^2}_{L^2(\Col,g)}}\leq C\ell^{1/2}$.
\end{lemma}
\begin{proof}[Proof of Lemma \ref{lemma:basis}]
We prove the lemma by contradiction.
Suppose instead that the lemma is false, and thus that there exist a closed surface $M$, a sequence of metrics $g_i$ on $M$, 
and a sequence of collars $\Col_i$ on $(M,g_i)$ corresponding to closed geodesics $\si_i$ of length $0<\ell_i\leq 2\arsinh(1)$ so that for each $i$, whenever $\Om$ is a holomorphic quadratic differential
on $(M,g_i)$ with $\norm{\Om}_{L^2(M,g_i)}=1$ then at least one of the two bounds
\beq \label{est:Om-complement2}
\norm{\Om}_{L^\infty(M\setminus \Col_i,g_i)}\leq  i\,\ell_i^{1/2},
\eeq
or
\beq \label{est:Om-collar2} 
\norm{\Om-b_0dz^2}_{L^\infty(\Col_i,g_i)}\leq i\, \ell_i^{1/2},
\text{ for some } b_0\in\C \text{ with }\abs{1-\abs{b_0}\cdot \norm{dz^2}_{L^2(\Col_i,g_i)}}\leq i\ell_i^{1/2}
\eeq
must be \textit{violated}.
To derive a contradiction, we thus construct elements $\Om_i$ which fullfill the two estimates \eqref{est:Om-complement2} and \eqref{est:Om-collar2} for $i$ sufficiently large.
 
We first remark that standard estimates for holomorphic functions on discs lead to an estimate of the form
\beq\label{est:holo-thick}
\norm{\Phi}_{{L^{\infty}(\de\thick(M,g))}}\leq C_\de\norm{\Phi}_{L^1(M,g)}\eeq
valid for all \textit{holomorphic} quadratic differentials $\Phi$ on a hyperbolic surface $(M,g)$ and every $\de>0$
with $C_\de$ depending only on $\de$ and the genus of the surface, cf. Lemma A.9 in \cite{R-T-Z}.
We conclude that the sequence of surfaces $(M,g_i)$ introduced above must degenerate as $i\to \infty$; 
indeed, assume that (for a subsequence) the injectivity radius of $(M,g_i)$ is bounded away from zero by some number $\de>0$, and thus that  
the length of any closed geodesic of $(M,g_i)$ is no less than $2\de$. Then estimate \eqref{est:holo-thick} implies that for $i$ sufficiently large
\eqref{est:Om-complement2} and \eqref{est:Om-collar2} are both satisfied, say for $b_0=0$, for every $\Om\in\Hol(M,g_i)$ with $\norm{\Om}_{L^2}=1$, leading to a contradiction.

The sequence $(M,g_i)$ can thus be analysed with the Deligne-Mumford compactness theorem in the same way as earlier, 
collapsing $k$ collars $\Col_i^j=\Col(\ell_i^j)$, $\ell_i^j\to 0$ %(of which we may assume the first one, $\Col_i^1$, corresponds to the $\Col_i$ already considered) 
and yielding a limit $(\Si,h)$.

The Fourier decomposition of holomorphic quadratic differentials $\Phi$ on each hyperbolic collar $(\Col,g)$ 
\beq \Phi=\bigg(\sum_{n=-\infty}^\infty b_n e^{ns}\,e^{in\th}\bigg)\cdot dz^2, \quad b_n\in\C\eeq
gives an $L^2(\Col,g)$-orthogonal decomposition of each such $\Phi$
into its principal part $b_0(\Phi)dz^2$ and its \emph{collar decay} part $\om^\perp(\Phi) dz^2:=\Phi-b_0(\Phi)dz^2$ which,
by Lemma 2.2 and Remark 2.3 of \cite{R-T-Z} 
satisfies the key estimate
\beq \label{est-collar-decay}
\norm{\om^\perp(\Phi) dz^2}_{L^{\infty}(\de\thin(\Col,g))}\leq C\de^{-2}e^{-\frac\pi\de}\norm{\Phi}_{L^1(M,g)}.
\eeq

Following \cite{R-T-Z}, we can then define the subspaces 
%$W_i\subset \Hol(M,g_i)$ as in to be 
\beq \label{def:Wi} W_i:=\{\Th\in\Hol(M,g_i):\, b_0^{j}(\Th)dz^2=0 \text{ for every } j\in\{1\ldots k\}\,\}\eeq
of all holomorphic quadratic 
differentials with principal part equal to zero on each degenerating collar $\Col_i^j$, $j=1\ldots k$, and it is these 
subspaces $W_i$ which converge to $\Hol(\Si,h)$ in the sense of Lemma \ref{lemma:convergence} (as described in \cite{R-T-Z}). We furthermore remark that 
elements of $W_i$ are uniformly controlled by their $L^2$ norm, 
\beq\label{Wi-Linfty} \sup_{w\in W_i}\frac{\norm{w}_{L^\infty(M,g_i)}}{\norm{w}_{L^2(M,g_i)}}\leq C<\infty\eeq
for a constant $C$ independent of $i$ (as follows from Lemma 2.4(i) and Lemma A.8 of \cite{R-T-Z}).

This implies in particular that the collar $\Col_i$, for which \eqref{est:Om-complement2} and \eqref{est:Om-collar2} cannot be satisfied, must degenerate, $\ell_i\to 0$ as $i\to \infty$, 
and thus that this collar coincides with one of the
collapsing collars $\Col_i^j$, 
say $\Col_i=\Col_i^1$ (for a subsequence).

We will now choose the holomorphic quadratic differentials $\Om_i$ associated with these collars as elements of the $L^2(M,g_i)$-orthogonal complement $W_i^{\perp}$ of $W_i$. 
More precisely, by Lemma 2.4 of \cite{R-T-Z}, 
we have $\dim(W_i^{\perp})=k$ for large enough $i$, so
we can assign to $\Col_i=\Col_i^1$ the unique element $\Om_i$ of $W_i^\perp$ with $\norm{\Om_i}_{L^2(M,g_i)}=1$ for which 
the principal part $b_0^{j}(\Om_i)dz^2$ on $\Col_i^{j}$ is equal to zero if $j\neq 1$ but is 
$b_0(\Om_i)dz^2$ for some $b_0(\Om_i)>0$ if $j=1$. We then claim that 
\beq \label{est:L1-Omij}
\lambda_i:=\norm{\Om_i}_{L^1(M,g_i)}\leq C\cdot [\ell_i]^{1/2}\eeq
and remark that this claim combined with \eqref{est:holo-thick} and \eqref{est-collar-decay} % and the orthogonality of principal and collar-decay part if we want a better estimate for b which I don't think we care about
directly implies \eqref{est:Om-complement2} and \eqref{est:Om-collar2}, thus giving the contradiction that proves
Lemma \ref{lemma:basis}. 

In order to prove the bound \eqref{est:L1-Omij} we now consider the sequence $\widetilde{\Om}_i:= (\lambda_i)^{-1}\cdot\Om_i$ of holomorphic quadratic differentials normalised to have
$L^1$ norm $\norm{\widetilde\Om_i}_{L^1(M,g_i)}=1$ and prove that the only part of $\widetilde{\Om}_i$ whose contribution to the $L^1$ norm does not vanish as $i\to \infty$ is its 
principal part on $\Col_i$.

To start with, Lemma \ref{lemma:convergence} allows us to extract a subsequence of $\widetilde{\Om}_i$ so that 
$f_i^*\widetilde\Om_i$ converges smoothly locally to a limit 
$\widetilde\Om_\infty\in\Hol(\Si,h)$, which must indeed be identically zero since by construction 
$P_{g_i}^{W_i}\widetilde\Om_i=0$ 
and thus according to Lemma \ref{lemma:projection} also $P_h^{\Hol(\Si,h)}(\widetilde\Om_\infty)=0$.
%, and so 
%$\widetilde\Om_\infty\equiv 0$. 
We conclude that
$$0=\norm{\widetilde\Om_\infty}_{L^1(\Si,h)}=1-\inf_{\de>0}\lim_{i\to \infty}\norm{\widetilde\Om_i}_{L^1(\de\thin(M,g_i))}$$
which means that for any $\de>0$ all of the $L^1$ norm of $\widetilde\Om_i$ concentrates in the limit $i\to \infty$ on the $\de$-thin part of $(M,g_i)$.

We observe that for $\de>0$ sufficiently small, the $\de$-thin part of $(M,g_i)$ is given as the union of the $\de$-thin parts of the degenerating collars $\Col_i^{j}$, but that 
estimate \eqref{est-collar-decay} implies 
$\norm{\widetilde\Om_i}_{L^1(\de\thin(\Col_i^{j}))}\leq C \de^{-2}e^{-\frac\pi\de}\to 0$ as $\de\searrow 0$ for each $j\neq 1$.

Meanwhile, \eqref{est-collar-decay} applied to $\Col_i^{1}=:\Col_i$ shows that the contribution of the collar decay part 
$\om^\perp(\widetilde\Om_i) dz^2$ 
of 
$\widetilde\Om_i$ on $\Col_i$ to the total $L^1$ norm of $\norm{\widetilde \Om_i}_{L^1(M,g_i)}=1$ vanishes in the limit $i\to \infty$. This 
means that the only remaining part of $\widetilde\Om_i$, namely the principal part $b_0(\widetilde\Om_i)\cdot dz^2=(\lambda_i)^{-1}b_0(\Om_i)\cdot dz^2$ of $\widetilde \Om_i$ on this one collar $\Col_i$, 
must have $L^1$ norm converging to $1$ as $i\to \infty$. 
Since  
$\norm{dz^2}_{L^1(\Col_i)}=8\pi\cdot X(\ell_i)\leq C\cdot [\ell_i]^{-1}$ 
we thus get an upper bound of 
$$\norm{\Om_i}_{L^1(M,g_i)}=:\lambda_i\leq C\cdot [\ell_i]^{-1} b_0(\Om_i)$$
for the $L^1$ norm of the original holomorphic quadratic differential. But with $\Om_i$ normalised to have $\norm{\Om_i}_{L^2}=1$ and with the 
principal and collar decay part being $L^2$-orthogonal we also know that 
$\norm{ b_0(\Om_i)\cdot dz^2}_{L^2(\Col_i,g_i)}\leq 1$ which, according to \eqref{sizes_on_collars}, means that $b_0(\Om_i)\leq C\cdot[\ell_i]^{3/2}$. The claim \eqref{est:L1-Omij} now follows.
\end{proof}

In summary, from the previous lemma, its proof and the analysis of the spaces $W_i$ carried out in \cite{R-T-Z}, 
we obtain the following general description of the spaces of 
holomorphic quadratic differentials on degenerating hyperbolic surfaces.
\begin{cor}
Let $(M,g_i)$ be a sequence of hyperbolic surfaces degenerating to a punctured hyperbolic surface 
$(\Si,h)$ by collapsing $k$ collars $\Col_i^j$. Then, for $i$ sufficiently large, the space of holomorphic quadratic differentials $\Hol(M,g_i)$
splits into
\begin{enumerate}
 \item[(i)] the $3(\gamma-1)-k$ dimensional subspace $W_i$ defined in \eqref{def:Wi} which converges to the space $\Hol(\Si,h)$ of 
$L^1$ holomorphic quadratic differentials on the limit surface as described in Theorem 2.6 of \cite{R-T-Z}, and
\item[(ii)] its orthogonal complement $W_i^\perp$, a basis of which is given by holomorphic quadratic differentials $(\Om_i^j)_{j=1}^k$ concentrating solely on the degenerating collars
$\Col_i^j$ as described in Lemma \ref{lemma:basis}.
\end{enumerate}
\end{cor}

\begin{proof}[Proof of Lemma \ref{lemma:alpha}]
Let $(M,g)$ be a closed hyperbolic surface and $\Col=\Col(\ell)$ a collar around a closed geodesic in $(M,g)$. Without loss of generality, we may assume that $\ell\leq 2\arsinh(1)$, and can apply Lemma \ref{lemma:basis} to obtain the corresponding holomorphic quadratic differential $\Om$.
To prove \eqref{est:alpha-lemma}
it is enough to consider collars around geodesics of small length $\ell$, in particular small enough so that the number $b_0$,
as in Lemma \ref{lemma:basis} characterising the principal part of $\Om$ on $\Col$, satisfies
$|b_0|\geq \norm{dz^2}_{L^2(\Col)}^{-1}(1-C\ell^{1/2})
\geq c\ell^{3/2}$ for some universal $c>0$, compare \eqref{sizes_on_collars}. 

Given any quadratic differential $\Psi\in\Qu(M,g)$ that is orthogonal to $\Om$, we combine the relation $\langle \Psi, \Om\rangle_{L^2(M,g)}=0$ with this bound on $b_0$ and with \eqref{sizes_on_collars} to find that 
\beqa 
\label{est:orthog} 
\ell^{3/2}\abs{\int_{-X}^{X}\int_{S^1}\psi \cdot \rho^{-2}d\th ds}&\leq C\cdot \abs{\langle\Psi,b_0 dz^2\rangle_{L^2(\Col,g)}}\\
&\leq C\left(\abs{\langle\Psi,\Om-b_0 dz^2\rangle_{L^2(\Col,g)}}+\abs{\langle\Psi,\Om\rangle_{L^2(M\setminus \Col,g)}}\right)\\
&\leq C\left(\norm{\Om-b_0dz^2}_{L^{\infty}(\Col,g)}+\norm{\Om}_{L^\infty(M\setminus \Col,g)}\right)\cdot \norm{\Psi}_{L^1(M,g)}\\
&\leq C\ell^{1/2}\cdot \norm{\Psi}_{L^1(M,g)},
\eeqa 
or equivalently that the mean values $\alpha(s)$ are small on average in the sense that they satisfy 
\beq\label{hamster}
\abs{\int_{-X(\ell)}^{X(\ell)}\alpha(s)\rho^{-2}(s) ds}\leq C \ell^{-1}\norm{\Psi}_{L^1(M,g)}.  %,$ $X=X(\ell).
\eeq

Note that if $\Psi$ were holomorphic, then the function $s\mapsto \alpha(s)$ would be constant
and \eqref{hamster} would imply
$\abs{\alpha}\leq C\ell^{2}\norm{\Psi}_{L^1}$ and thus in particular the estimate of Lemma \ref{lemma:alpha}.
For general quadratic differentials $\Psi\in\Qu(M,g)$ the function $s\mapsto\alpha(s)$ need not be constant but we can still estimate
\beqa \label{est:alpha-diff} \abs{\alpha(0)-\alpha(s_0)} 
&= \abs{\frac1{2\pi}\int_0^{s_0}\frac{d}{ds}\left(\int_{\{s\}\times S^1} \psi d\th\right) ds}= \abs{\frac1{2\pi}\int_0^{s_0}\int_{S^1}(\partial_s\psi +i\partial_\th \psi)\,d\th ds}\\
&\leq \frac1\pi \int_{[0,s_0]\times S^1}\abs{\pbz \psi} d\th ds
\eeqa
for each $s_0\in(-X,X)$, where we abuse notation by allowing 
$[0,s_0]$ to denote $[s_0,0]$ for $s_0<0$. 
Using \eqref{sizes_on_collars}, we then write 
$$\alpha(0)\cdot \norm{dz^2}_{L^2(\Col,g)}^2=8\pi\int_{-X}^X(\alpha(0)-\alpha(s_0))\rho^{-2}(s_0) ds_0+4\int_{-X}^{X}\int_{S^1}\psi(s,\th) \rho^{-2}(s)d\th ds$$
and use \eqref{sizes_on_collars}, \eqref{est:orthog} and \eqref{est:alpha-diff} 
to estimate
\beqa \label{est:alpha0}
\abs{\alpha(0)}&\leq C\ell^3 \left[\int_0^X\bigg( \rho^{-2}(s_0) \int_{-s_0}^{s_0}\int_{S^1}\abs{\pbz \psi} d\th ds\bigg) ds_0 + \ell^{-1}\norm{\Psi}_{L^1(M,g)}\right]\\
&\leq C \ell^3\left(\int_{-X}^X\int_{S^1}\rho^{-1}\abs{\pbz \psi} d\th ds\right)\cdot \left(\int_0^X\rho^{-1}(s_0) ds_0\right)  +C \ell^2\norm{\Psi}_{L^1(M,g)}\\
&\leq C \ell \norm{\bar \partial \Psi}_{L^1(M,g)}+C \ell^2\norm{\Psi}_{L^1(M,g)}.
\eeqa
Here we use that 
\beqs%\label{eq:L1-pbz} 
\abs{\bar \partial \Psi}=\abs{\pbz \psi}\cdot \abs{d\bar z \tensor dz^2 }=2\sqrt{2} \abs{\pbz \psi}\rho^{-3}\eeqs
so that
\beqs %\label{est:psi-bar}
\norm{\bar \partial \Psi}_{L^1(\Col,g)}=2\sqrt{2}\int_{-X}^{X}\int_{S^1}\rho^{-1}(s)\abs{\pbz \psi(s,\th)} d\th ds.\eeqs
We also used that $\rho(s)$ is monotone in $\abs{s}$ and that $\int_0^{X}\rho^{-1}(s_0) ds_0=\frac{2\pi}{\ell}\int_0^{X(\ell)}\cos\left(\frac{\ell s}{2\pi}\right) ds\leq \left(\frac{2\pi}{\ell}\right)^2$.
Combining \eqref{est:alpha0} with \eqref{est:alpha-diff} we thus find that for each $s_0\in(-X(\ell),X(\ell))$ 
\beq\label{est:alpha2} \abs{\alpha(s_0)}\leq C\ell\ \norm{\bar \partial \Psi}_{L^1(M,g)}+\frac1\pi\int_{[0,s_0]\times S^1}\abs{\pbz \psi} d\th ds+C\cdot \ell^2\norm{\Psi}_{L^1(M,g)}
.\eeq
We stress that the second term on the right-hand side of this estimate is \textit{not} the $L^1$ norm of $\bar \partial \Psi$ over $[0,s_0]\times S^1$ but a 
much smaller integral; indeed the missing factor $\rho^{-1}(s)$ controls the (euclidean) distance of $s$ to the end of the collar since
\beq\label{est:rho} 
\rho(s)\cdot (X(\ell)-\abs{s})\leq\rho(s)\cdot\left(\frac{\pi^2}{\ell}-\abs{s}\right)\leq \sup_{v\in(0,\pi/2)}\frac{v}{\sin(v)}
=\frac{\pi}{2}.\eeq

Integrating \eqref{est:alpha2} using Fubini's theorem as well as $X=X(\ell)\leq \frac{\pi^2}{\ell}$ 
we thus find 
\beqa \int_{-X}^{X}\abs{\alpha(s_0)} ds_0&\leq C\cdot \norm{\bar \partial \Psi}_{L^1(M,g)}+C\int_{-X}^X\int_{S^1}\abs{\pbz \psi}\cdot (X-\abs{s}) d\th ds+C\ell\norm{\Psi}_{L^1(M,g)}\\
&\leq C\norm{\bar \partial \Psi}_{L^1(M,g)}+C\ell\norm{\Psi}_{L^1(M,g)} \label{est:alpha}
\eeqa 
as claimed in Lemma \ref{lemma:alpha}.
\end{proof}

The remaining step in the paper is thus:
\begin{proof}[Proof of Lemma \ref{lemma:L1-est-using-alpha}]
We want to estimate the $L^1$ norm of a general quadratic differentials $\Psi\in\Qu(M,g)$ on the $\de$-thin part of a hyperbolic collar $\Col(\ell)$. 
The basic idea is to extend the function $\psi$ representing $\Psi=\psi dz^2$ on the collar periodically (with period $2\pi$ in the $i$ direction) to a function (still denoted by $\psi$) 
on the set $(-X(\ell),X(\ell))\times \R\subset \R^2\cong \C$ 
and to derive estimates using the inhomogeneous Cauchy-formula on large domains.
Before we proceed with the proof, we remark that the estimate of Lemma \ref{lemma:L1-est-using-alpha} is trivially true for large values of $\de$ so that we may henceforth assume that 
$0<\de<\de_0$ and thus also $0<\ell<2\de_0$ for a small number $\de_0>0$ to be chosen later on. 

Recall that for any $z_0\in \C$, any domain $\Om\subset \C$ containing $z_0$ and with piecewise $C^1$ boundary $\partial \Om$, and any $C^1$ function $\psi$
the Cauchy-formula gives 
\beq \label{Cauchy}
\psi(z_0)=\frac1{2\pi i}\int_{\Omega}\frac{\pbz \psi}{z-z_0} dz\wedge d\bar z
+\frac1{2\pi i}\int_{\partial \Om}\frac{\psi}{z-z_0} dz. \eeq

Keeping in mind the final goal of getting a bound on $\norm{\Psi}_{L^1(\de\text{-thin}(\Col))}$ in terms of $\norm{\bar\partial\Psi}_{L^1}$, $\alpha$ and 
a \textit{small} multiple of $\norm{\Psi}_{L^1}$ we would like to choose the domains $\Om$ in such a way that the boundary integrals in \eqref{Cauchy} are essentially given in terms of
the mean values $\alpha(\cdot)$.
Working on \textit{large} rectangles this can be easily achieved for the integrals along lines in the $\th$ direction. Furthermore, applying 
the Cauchy-formula not just for one such rectangle, but rather taking its mean value over a suitable family of rectangles, also the integrals along lines in the $s$ direction will be essentially 
controlled in terms of $\alpha$. We are able to control the first integral in \eqref{Cauchy} in $L^1$ provided we choose the size of these rectangles dependent on the (large) factor $\rho^{-1}(s_0)$ 
with which $\pbz\psi$ appears in $\norm{\bar \partial \Psi}_{L^1}$.

Let now $\Col(\ell)\cong((-X(\ell),X(\ell))\times S^1,\rho^2(ds^2+d\theta^2))$ be a hyperbolic collar around a closed geodesic of length $0<\ell<2\de_0$. 
Then for each point $z_0=(s_0,\th_0)\in (-X_{\de_0},X_{\de_0})\times [0,2\pi]$, representing a point in the $\de_0$-thin part of the collar,
we consider the family of rectangles
$$\Om_b(z_0):=\{z_0\}+[-\rho^{-1/2}(s_0),\rho^{-1/2}(s_0)]\times [-(\rho^{-1/2}(s_0)+b),\rho^{-1/2}(s_0)+b],b\in[0,2\pi]$$
and apply the Cauchy-formula to write
\beq\label{eq:Cauchy-psi}
2\pi i \psi(z_0)=I_\Om(z_0,b)+I_V^{+}(z_0,b)+I_V^{-}(z_0,b)+I_H^{+}(z_0,b)+I_H^{-}(z_0,b)\eeq
where $I_\Om(z_0,b)=\int_{\Om_b(z_0)}\frac{\pbz \psi}{z-z_0} dz\wedge d\bar z$ while 
$I_H$, $I_V$ denote the line integrals along the horizontal respectively vertical paths of $\partial \Om_b$, that is
$$I_H^{\pm}(z_0,b)=\mp\int_{h^-(s_0)}^{h^+(s_0)}\frac{\psi(s,\th_0\pm(\rho^{-1/2}(s_0)+b))}{s-s_0\pm i(\rho^{-1/2}(s_0)+b)} ds$$
and 
$$I_V^{\pm}(z_0,b)=\pm i\int_{\th_0-(\rho^{-1/2}(s_0)+b)}^{\theta_0+\rho^{-1/2}(s_0)+b}\frac{\psi(h^{\pm}(s_0),\theta)}{\pm \rho^{-1/2}(s_0)+i(\theta-\theta_0)} d\theta.$$
Here and in the following we write for short $h^{\pm}(s_0):=s_0\pm\rho^{-1/2}(s_0)$ to denote the $s$ limits of the domains of integration. 

Remark that to obtain estimates on $\psi$, be it pointwise or in an $L^1$ sense, it is sufficient to prove estimates on the mean values with respect to $b$ of all these integrals. 
While for the terms $I_\Om$ and $I_V^{\pm}$ it is equally simple/difficult to derive bounds on these terms for each individual $b$, taking such an average over $b$ is crucial in order to 
bound the integrals $I_H^\pm$ along horizontal lines in terms of $\alpha$ which is a mean value in $\th$ and not in $s$.

With this in mind, we bound 
\beqa\label{est:Cauchy-psi}
\norm{\Psi}_{L^1(\de\thin(\Col,g))}  &=\frac{1}{\pi}\int_{-X_\de}^{X_\de}\int_{S^1}\abs{\fint_0^{2\pi} \big(I_\Om+I_V^{+}+I_V^{-}+I_H^{+}+I_H^{-}\big)(z_0,b) db} d\th_0ds_0\\
&\leq \frac{1}{\pi}\sup_{b\in[0,2\pi]} \int_{-X_\de}^{X_\de}\int_{S^1}\abs{I_\Om(z_0,b)}+\abs{I_V^{+}(z_0,b)}+\abs{I_V^{-}(z_0,b)} d\th_0ds_0\\
&+\frac1{\pi}\int_{-X_\de}^{X_\de}\int_{S^1} \abs{\fint_0^{2\pi}I_H^{+}(z_0,b) db}+\abs{\fint_0^{2\pi}I_H^{-}(z_0,b) db} d\th_0ds_0
\eeqa
and estimate all terms occurring in this formula individually. As remarked above, we may always assume that $0<\de<\de_0$ for a small fixed number $\de_0>0$, which we
chose in particular so that the following remark holds true.

\begin{rmk}\label{rem}
For $\de_0>0$ sufficiently small the domains $\Om_b$ have been chosen in such a way that for each collar $\Col=\Col(\ell)$ with $0<\ell<2\de_0$ and each  $0<\de\leq\de_0$:
\begin{enumerate}
 \item[(i)] For every $z_0\in\de\text{-thin}(\Col)$ the points in $\Om_b(z_0)$ all correspond to points in the $2\de$-thin part of $\Col$, i.e.~ 
$\Om_b(z_0)\subset (-X_{2\de},X_{2\de})\times \R$.
\item[(ii)] There is a constant $C=C_{\de_0}<\infty$ depending only on $\de_0$ such that 
$$C^{-1}\rho(s)\leq \rho(s_0)\leq C\rho(s) \text{ for all } z=(s,\th)\in \Om_b(z_0), z_0=(s_0,\th_0)\in\de_0\text{-thin}(\Col).$$
\item[(iii)] The functions $h^{\pm}(s)=s\pm\rho(s)^{-1/2}(s)$ are invertible on $[X_{-\de_0},X_{\de_0}]$, the derivatives of the inverses uniformly 
bounded and 
$((h^-)^{-1}-(h^+)^{-1})(s)\leq C_{\de_0}\rho^{-1/2}(s)$.
\end{enumerate}
\end{rmk}

The main observation leading to the first statement of the remark is that the expression \eqref{eq:Xdelta} for $X_\de$ 
combined with the mean value theorem 
implies that for small values of $\de$ the difference $X_{2\de}-X_{\de}\geq \frac{\pi\cdot\sinh(\ell/2)}{\ell\cdot \sinh(\de)}\geq c \de^{-1}$ 
is much larger than $\rho^{-1/2}(X_\de)\leq C \de^{-1/2}$, compare \eqref{est:rho-Xde}.
The second remark is then a simple consequence of the first and of \eqref{est:rho-Xde}. For the final claim, we observe that the derivative of $\rho^{-1}$ is 
uniformly bounded so that $(\rho^{-1/2})'\leq C \rho^{1/2}$ is small in the $\de_0$-thin part of the collar that we consider, and thus the derivatives of $h^{\pm}$ are close to one.

Turning back to \eqref{est:Cauchy-psi}, we first estimate the term involving the antiholomorphic derivative (i.e. involving $I_\Om$). 
Let $N(s_0):=\big[\frac{\rho^{-1/2}(s_0)}{2\pi}\big]$ and remark that the domain $\Om_b(z_0)$ 
can be wrapped around the cylinder $[h^-(s_0),h^+(s_0)]\times S^1$ no more than $2(N(s_0)+2)$ times.
Using the periodicity of $\psi$ we can thus estimate for each $b\in [0,2\pi]$
\beqa
\abs{I_{\Om}(z_0,b)}&\leq 2\int_{\Om_{2\pi}(z_0)}\abs{\frac{\pbz \psi}{z-z_0}} d\theta ds \\
&\leq 2\int_{z_0+[-2\pi,2\pi]^2} \abs{\frac{\pbz \psi}{z-z_0}} d\theta ds+2\int_{h^-(s_0)}^{h^+(s_0)}\int_{-2\pi}^{2\pi}\frac{\abs{\pbz\psi}}{2\pi} d\theta ds\\
&+4\sum_{k=2}^{N(s_0)+2} \frac{1}{2\pi(k-1)}\cdot 
\int_{h^-(s_0)}^{h^+(s_0)}\int_{S^1}\abs{\pbz\psi} d\th ds\\
&\leq 
2\int_{z_0+[-2\pi,2\pi]^2}\abs{\frac{\pbz \psi}{z-z_0}} d\theta ds+
C \cdot  
\int_{h^-(s_0)}^{h^+(s_0)}\int_{S^1}\abs{\pbz \psi} \cdot\log(\rho^{-1}(s))d\theta ds 
\eeqa
with the last inequality due to the bound $1+\sum_{k=1}^{N(s_0)+1}\frac1k\leq C\cdot \log(N(s_0)+1)\leq C\log(\rho^{-1}(s))$ being valid for every $s\in[h^-(s_0),h^+(s_0)]$, see Remark \ref{rem} (ii), if $\de_0$ is small enough.

Integrating this estimate over $z_0\in \de\text{-thin}(\Col)=(-X_\de,X_\de)\times S^1$ using Fubini's theorem as well as (i) and (iii) of Remark \ref{rem},  we thus obtain that for any $b\in [0,2\pi]$, we have
\beqa \label{est:I_Om}
\int_{\de\text{-thin}(\Col)}\abs{I_\Om(z_0,b)} ds_0d\th_0
&\leq  2\int_{-X_\de-2\pi}^{X_\de+2\pi}\int_{S^1}\bigg[ \abs{\pbz \psi} \cdot \int_{z+[-2\pi,2\pi]^2}\frac{1}{\abs{z-z_0}}ds_0d\th_0\bigg]d\theta ds \\
&\, +C \int_{-X_{2\de}}^{X_{2\de}}\int_{S^1} \bigg[\log(\rho^{-1})\cdot \abs{\pbz \psi}\cdot \big(\int_{(h^+)^{-1}(s)}^{(h^-)^{-1}(s)}1 ds_0\big)\bigg]d\theta ds \\
&\leq C\int_{-X_{2\de}}^{X_{2\de}}\int_{S^1}  \abs{\pbz \psi} \cdot \big(\rho^{-1/2}\log(\rho^{-1})+1\big)d\th ds\\
&\leq C\int_{-X_{2\de}}^{X_{2\de}}\int_{S^1}  \abs{\pbz \psi} \rho^{-1}d\th ds= C\norm{\bar \partial\Psi}_{L^1(2\de\thin(\Col))}.
\eeqa

Next we estimate the line integrals $I_V^{\pm}$ over vertical paths. Since we handle both terms in the same way, here we demonstrate the argument only by treating $I_V^{+}$.

For any $b\in [0,2\pi]$ we split $I_V^+(z_0,b)$ into integrals over the line segments 
$I_k=\{h^+(s_0)\}\times\big(\{\theta_0+2\pi\cdot k\}+[0,2\pi]\big)$, $\abs{k}\leq N(s_0)$, with $N(s_0)$ as above,  
and a small remainder term that is bounded by $C\rho^{1/2}(s_0)\int_{S^1}\abs{\psi(h^+(s_0),\theta)} d\theta$.

The important observation is that $\frac{1}{z-z_0}$ is nearly constant over each such $I_k$ and consequently that the corresponding integrals are essentially given by multiples of the mean values
$$\alpha(h^+(s_0))=\fint_{S^1}\psi(h^+(s_0),\th)d\th=\fint_{I_k}\psi.$$
More precisely, we have
\beq\label{est:error-mv} \sup_{z\in I_k}\textstyle{\abs{\frac1{z-z_0}-\frac1{\rho^{-1/2}(s_0)+2\pi i\cdot k}}}\leq 2\pi\rho(s_0),\eeq
so that for each $k$
\beqa 
\abs{\int_{I_k}\frac{\psi(z)}{z-z_0}dz}
&\leq 2\pi \rho^{1/2}(s_0)\cdot \abs{\alpha(h^+(s_0))} +C\rho(s_0)\cdot \int_{S^1}\abs{\psi(h^+(s_0),\theta)} d\theta.
\eeqa
Summing up these $2N(s_0)+1\leq C\rho^{-1/2}(s_0)$ integrals, we thus find that for each $b\in[0,2\pi]$
$$\abs{I_V^{+}(z_0,b)}\leq C\cdot \abs{\alpha(h^+(s_0))}+C\rho^{1/2}(s_0)\cdot \int_{S^1}\abs{\psi(h^+(s_0),\theta)} d\theta.$$

Integrating these estimates over $z_0$ in the $\de$-thin part of the collar and using Remark \ref{rem} (i) and (iii) as well as \eqref{est:rho-Xde} 
we thus find that for any $b\in[0,2\pi]$
\beqa \label{est:I_V}
\int_{\de\text{-thin}(\Col)}\abs{I_V^+(z_0,b)} ds_0d\th_0&\leq C\int_{-X_{2\de}}^{X_{2\de}}\abs{\alpha(s)} ds+C\left(\sup_{s\in[-X_\de,X_\de]}\rho^{1/2}(s)\right)
\cdot \norm{\psi}_{L^1(2\de\thin(\Col))}\\
&\leq C
\int_{-X}^X\abs{\alpha(s)} ds+C\delta^{1/2}\norm{\Psi}_{L^1(M,g)}.
\eeqa

We finally derive estimates for the integrals $I_H^{\pm}(z_0,b)$ over the horizontal paths, now not for each individual $b$ but rather for the mean values over $b\in [0,2\pi]$.
We treat the term $I_H^+$, with $I_H^-$ handled similarly. 
By Fubini's theorem
$$\abs{\fint_0^{2\pi}I_H^+(z_0,b) db}\leq \int_{h^-(s_0)}^{h^+(s_0)}\abs{\fint_{\{\th_0+\rho^{-1/2}(s_0)\}+[0,2\pi]}\frac{\psi(s,\th)}{z-z_0} d\th} ds.$$

Using an estimate similar to \eqref{est:error-mv} we now write the interior integrals as $\frac{\alpha(s)}{s-s_0+i\rho^{-1/2}(s_0)}$ plus a small error term, resulting in
\beqa\label{est:IH2}
\abs{\fint_0^{2\pi}I_H^+(z_0,b) db}&\leq \rho^{1/2}(s_0)\cdot \int_{h^-(s_0)}^{h^+(s_0)} \abs{\alpha(s)}ds+C\rho(s_0)\int_{h^-(s_0)}^{h^+(s_0)} \int_0^{2\pi}\abs{\psi} d\th ds\\
&\leq  C\cdot \int_{h^-(s_0)}^{h^+(s_0)} \rho^{1/2}(s)\abs{\alpha(s)}ds+C\rho^{1/2}(s_0)\int_{h^-(s_0)}^{h^+(s_0)} \int_0^{2\pi}\rho^{1/2}(s)\abs{\psi} d\th ds,
\eeqa
where we once more use Remark \ref{rem} (ii) in the last step. Thanks to Remark \ref{rem}, after integration over $z_0\in\de\thin(\Col)$ this gives
\beqa\label{est:I_H}
\int_{-X_\de}^{X_\de}\int_{S^1}\abs{\fint_0^{2\pi}I_H^+(z_0,b)db} d\theta_0ds_0&\leq C\int_{-X_{2\de}}^{X_{2\de}}\abs{\alpha(s)} \rho^{1/2}(s)\cdot \big((h^-)^{-1}-(h^+)^{-1}\big)(s) ds\\
&+C\de^{1/2} \int_{-X_{2\de}}^{X_{2\de}}\int_{S^1}\abs{\psi(s,\theta)} \rho^{1/2}(s)\cdot ((h^-)^{-1}-(h^+)^{-1})(s)d\theta ds \\
&\leq C\int_{-X}^X\abs{\alpha(s)}ds+C\de^{1/2}\norm{\Psi}_{L^1(M,g)}.
\eeqa

Inserting these three estimates \eqref{est:I_Om}, \eqref{est:I_V} and \eqref{est:I_H} into \eqref{est:Cauchy-psi} immediately gives the claim of Lemma \ref{lemma:L1-est-using-alpha}. 
\end{proof}

{\sc Max-Planck-Institut f\"ur Gravitationsphysik, Am M\"uhlenberg 1, 14476 Golm, Germany}

{\sc Mathematics Institute, University of Warwick, Coventry,
CV4 7AL, UK}


\begin{thebibliography}{99}

\bibitem{Del-Mum} P. Deligne and D. Mumford, \textit{The irreducibility of the space of curves of given genus},
 Inst. Hautes \'{E}tudes Sci. Publ. Math. \textbf{36} (1969), 75--109.

%\bibitem{Ha} N. Halpern, \textit{A proof of the collar lemma},
%Bull. London Math. Soc. \textbf{13} (1981), 141--144.

\bibitem{Hu} C. Hummel, \textit{Gromov's compactness theorem for
pseudo-holomorphic curves}, Progress in Mathematics, \textbf{151},
Birkh\"{a}user Verlag, Basel, (1997), viii+131 pp.

\bibitem{mumford} D. Mumford, \emph{A Remark on Mahler's Compactness Theorem.\/} Proc. Amer. Math. Soc. {\bf 28} (1971) 289--294.

\bibitem{randol} B. Randol \emph{Cylinders in Riemann surfaces.\/}
Comment. Math. Helvetici \textbf{54} (1979) 1--5.

\bibitem{R-T} M. Rupflin and P.M. Topping, \textit{Flowing maps to minimal surfaces}, preprint, \url{http://arxiv.org/abs/1205.6298}

\bibitem{R-T-Z} M. Rupflin and P.M. Topping and M. Zhu, \textit{Asymptotics of the Teichm\"uller harmonic map flow}, preprint, \url{http://arxiv.org/abs/1209.3783}

\bibitem{annals} P.M. Topping, \textit{Repulsion and quantization in almost-harmonic maps, and asymptotics of the harmonic map flow.} Annals of Math., {\bf 159} (2004) 465--534.

\bibitem{tromba}  A. Tromba, \textit{Teichm\"uller theory in Riemannian geometry}, Lecture notes prepared by Jochen Denzler. Lectures in Mathematics ETH-Z\"urich. Birkh\"auser (1992).

\bibitem{Bu} P. Buser: Geometry and spectra of compact Riemann surfaces. Progress
in Mathematics, \textbf{106}, (Birkh\"{a}user Boston, Inc., Boston,
MA, 1992), Ch.4

\end{thebibliography}
\end{document}